\documentclass{article}
\usepackage{amsmath}
\usepackage{amsthm}
\usepackage{amsfonts}
\usepackage{amssymb}
\usepackage{amscd}
\usepackage[all]{xy}
\usepackage{graphicx}
\usepackage{xcolor}
\usepackage{enumerate}

\font\smallsl=cmsl10

\usepackage{comment}
\usepackage{hyperref}

\newtheorem{theorem}{Theorem}

\newtheorem{corollary}[theorem]{Corollary}

\newtheorem{lemma}[theorem]{Lemma}

\newtheorem{proposition}[theorem]{Proposition}

\newcommand{\im}{\mathrm {im\,}}

\newcommand{\diag}{{\rm diag}}

\newcommand{\ee}{\quad\text{ and }\quad}

\renewcommand{\tilde}{\widetilde}

\usepackage[
    backend=biber,
    style=numeric,
	sorting=nyt,
    natbib=true,
    url=false, 
    doi=true,
    backref=true
    eprint=false
]{biblatex}
\begin{document}

\title{Type of homomorphisms of complex tori}
\author{Juliana Coelho\\{\scriptsize julianacoelhochaves@id.uff.br}
}

\maketitle

\begin{abstract}
Using determinantal divisors of integral matrices, that is, 
greatest common divisors of minors of fixed order,
we introduce the notion of type of a homomorphism $f$ of complex tori, 
which is similar to the type of a polarization.
We show that the type is invariant under composition with isomorphisms, and that it
completely describes the kernel of $f$ as a group.
More precisely, 
the quotient of the kernel by its connected component
containing $0$ is a
product of cyclic groups whose orders are determined by the type of $f$.
Since the type can be computed from any rational representation of $f$, this gives 
an effective way to determine the kernel of a homomorphism.
As a consequence, we compute the classic invariants degree and exponent of $f$, when $f$ has finite kernel. 
When $f$ is an isogeny, we also compute the type of its inverse isogeny from that of $f$.
Finally, we compare the type of a polarization to the type of its associated isogeny.
\end{abstract}


\section{Introduction}


A \emph{complex torus} of dimension $g$ is a quotient $T=V/L$ of a complex vector space $V$ of dimension $g$ and a lattice $L$ on $V$, that is, a discrete subgroup of $V$ of rank $2g$ generating $V$ as a real vector space. 
Given bases $v_1,\ldots,v_g$ of $V$ and $l_1,\ldots,l_{2g}$ of $L$, the \emph{period matrix} of $T$ is the matrix $\Pi$ whose $j$-th column consists of the coefficients of the expression of $l_j$ in the basis $v_1,\ldots,v_g$. 

For complex tori $T_1=V_1/L_1$ and $T_2=V_2/L_2$, a \emph{homomorphism} $f\:T_1\rightarrow T_2$ is induced by a linear map $\rho_a(f)\:V_1\rightarrow V_2$ called the \emph{analytic representation} of $f$. 
The restriction of $f$ to $L_1$ gives a group homomorphism $\rho_r(f)\:L_1\rightarrow L_2$ called the \emph{rational representation} of $f$. Fixing bases for $V_j$ and $L_j$ for $j=1,2$, we may consider the matrices $[\rho_a(f)]$ and $[\rho_r(f)]$ associated to the analytic and rational representations of $f$. 
If $\Pi_1$ and $\Pi_2$ are the period matrices of $T_1$ and $T_2$ with respect to the chosen bases, the Hurwitz or matricial equation of $f$ is given by 
$$[\rho_a(f)]\Pi_1=\Pi_2[\rho_r(f)].$$

The kernel $\ker f$ of the homomorphism $f$ is a compact subgroup of $T_1$. 
If $\ker f$ is finite, we define the \emph{degree} $\deg(f)$  of $f$ as its order and
if $\ker f$ is not finite,  then we set $\deg(f):=0$. 
In any case, $\ker f$ has a finite number of connected components,
which are all translates of the connected component $(\ker f)_0$ containing 0.
Moreover, $(\ker f)_0$ is a subtorus of $T_1$ and thus a subgroup of $\ker f$, and
the number of connected components of $\ker f$ is the index of $(\ker f)_0$ in $\ker f$.
If $f$ has finite kernel, we also define the \emph{exponent} $e(f)$ of $f$ as the exponent of the finite group $\ker f$, 
that is, $e(f)$ is the smallest $n\in\mathbb Z_{>0}$ such that  
$nt=0$ for all $t\in\ker f$.

The main goal of this paper is to give a procedure to describe the kernel of a homomorphism $f$ of complex tori and, in particular, to compute its degree and exponent.
To tackle this problem 
we use minors of the matrix of its rational representation.
We define the \emph{determinantal divisor} $\delta_k(f)$ of $f$ to be the greatest common divisor of the minors of order $k$ of $[\rho_r(f)]$, when at least one of these minors is non-zero, and we set $\delta_k(f)=0$ otherwise. 
The determinantal divisors are independent of the choice of bases and are invariant under isomorphism 
(Proposition \ref{prop:smith-normal-form-homo}). 
If $r$ is the rank of $[\rho_r(f)]$, we may define 
 $$d_k(f):=\delta_k(f)/\delta_{k-1}(f)$$ 
for $1\leq k\leq r$, where $\delta_0(f):=1$.
Then $d_k(f)\in \mathbb Z$ and we define  the \emph{type} of $f$ as $(d_1(f),\ldots,d_r(f))$.
Clearly the type is also 
invariant under isomorphism, since this holds for the determinantal divisors. 

The type of a homomorphism
is a much finer invariant than the degree and the exponent, and one that encompasses both concepts.
More precisely, we show that the type of $f$ completely determines its kernel (Theorem \ref{teo:type-detrmines-kernel}). 
As a consequence, it can be used to compute the degree and the exponent of $f$, and also the number of connected components of $\ker f$ (Corollaries \ref{cor:grauisog} and \ref{cor:exponent}). 
Summing up, we show:

\smallskip

\noindent{\bf Theorem.} {\it Let $f\:T_1\rightarrow T_2$ be a homomorphism of complex tori.
Then 
$$\frac{\ker f}{(\ker f)_0}\cong \mathbb Z_{d_1(f)}\times \cdots \times \mathbb Z_{d_{2(g_1-h)}(f)},
$$
where 
$g_1=\dim T_1$ and $h=\dim(\ker f)_0$.  

As a consequence, 
the number of connected components of $\ker f$ is equal to $\delta_{2(g_1-h)}(f)$.
Moreover, if $\ker f$ is finite then
$$deg(f)=\delta_{2g_1}(f) \ee e(f)=d_{2g_1}(f).$$
}

\smallskip

In addition, let $f\:T_1\rightarrow T_2$ be an \emph{isogeny}
and $f'\:T_2\rightarrow T_1$  be its
\emph{inverse isogeny}, that is,
$f$ is a surjective homomorphism with finite kernel
and $f'$ satisfies
$f\circ f'=e(f)$ and $f'\circ f=e(f)$.
Then we show how to compute the type of $f'$ from that of $f$ (Corollary \ref{cor:exponent}). 
Furthermore, 
we present an alternative way of computing the exponent of an isogeny and, as a consequence, we show that the exponent of a composition $g\circ f$ of isogenies divides the product of the exponents of $g$ and $f$ (Corollaries \ref{cor:expoente2} and \ref{cor:composition-exponent}).

Finally, recall that a \emph{polarization} on a complex torus $T=V/L$ is a positive definite hermitian form $H$ whose imaginary part $E$ has integral values on $L$.
We show how the classical type of the polarization $H$ can be computed from the type of the isogeny associated to it (Proposition \ref{prop:type-polarization}).

%
%
%
%
%

\section{Determinantal divisors of integral matrices}\label{sec:divisors-matrices}

An \emph{integral matrix} is a matrix with entries in  $\mathbb Z$.
We denote by $M(n\times m,\mathbb Z)$ the set of integral matrices of order $n\times m$,
and by $GL_n(\mathbb Z)$ the subring of $M(n\times n,\mathbb Z)$ consisting of invertible integral matrices, that is, matrices with integral entries having determinant $\pm 1$. 

For $A\in M(n\times m,\mathbb Z)$ and $k\in\mathbb Z_{>0}$,
the \emph{$k$-th determinantal divisor} $\delta_k(A)$ is the greatest common divisor of the minors of order $k$ of $A$, if at least one of these minors is non-zero. If all minors of order $k$ of $A$ are zero or if $k>\min\{n,m\}$, we set $\delta_k(A)=0$. 
Note that $\delta_1(A)$ is the greatest common divisor of the entries of $A$ and that, if $n=m$, then 
$\delta_n(A)=|\det(A)|$.


The following results are standard in the theory of integral matrices and are 
included here for the sake of completeness.
They can be found, for instance, in \cite[Chapter II]{newman}.

\begin{lemma}\label{lem:determinantal-divisors}
Let  $A\in M(n\times m ,\mathbb Z)$. 
\begin{enumerate}[(a)]
\item If $\delta_k(A)=0$ then $\delta_{k+1}(A)=0$. 
If $\delta_k(A)\neq 0$, then $\delta_k(A)$ divides $\delta_{k+1}(A)$. 
In particular if $A\in GL(n,\mathbb Z)$ then $\delta_k(A)=1$ for every $1\leq k\leq n$.

\item Let $B\in M(m\times p,\mathbb Z)$. 
If $\delta_k(A)=0$ or $\delta_k(B)=0$ then $\delta_k(A\cdot B)=0$. 
If $\delta_k(A)\neq 0$ and $\delta_k(B)\neq 0$ then $\delta_k(A)\cdot \delta_k(B)$ divides $\delta_k(A\cdot B)$. 

\item If
$B\in GL(n,\mathbb Z)$ 
and
$B'\in GL(m,\mathbb Z)$ 
then 
$\delta_k(A)=\delta_k(B\cdot A\cdot B')$. 
\end{enumerate}
\end{lemma}
\begin{proof}
Follows from the determinant and Cauchy-Binet formulas (see for instance \cite[Section I.4]{gantamacher}).
\end{proof}

Note that, unlike the determinant, the determinantal divisors are usually not multiplicative. Indeed, for
$$A=\left(\begin{array}{cc}
1&0\\0&2
\end{array}\right)
\ee
B=\left(\begin{array}{cc}
2&0\\0&1
\end{array}\right),$$
we have $\delta_1(A)=\delta_1(B)=1$ but $\delta_1(A\cdot B)=2$. 
\smallskip

We say that two matrices $A_1,A_2\in M(n\times m,\mathbb Z)$ are \emph{equivalent} if there exist
$B\in GL(n,\mathbb Z)$ 
and
$B'\in GL(m,\mathbb Z)$ 
such that $A_2=B\cdot A_1\cdot B'$. By the previous lemma, equivalent matrices have the same determinantal divisors.
The converse follows from the Smith normal form. 

\smallskip

For $d_1,\ldots,d_r\in\mathbb Z$, denote by $\diag(d_1,\ldots,d_r)$ the diagonal matrix with main diagonal given by $d_1,\ldots,d_r$.

\begin{lemma}\label{lem:smith-normal-form} \emph{(Smith normal form)}
Let $A\in M(n\times m ,\mathbb Z)$ of rank $r$. Then there exist unique
$d_1,\ldots,d_r\in\mathbb Z_{>0}$ with $d_k|d_{k+1}$ 
such that $A$ is equivalent to a matrix of the form
\begin{equation}\label{eq:smith-normal-form-matrix}
\left(\begin{array}{cc}
D & 0\\ 0 & 0
\end{array}\right),
\end{equation}
where $D=\diag(d_1,\ldots,d_r)$.
\end{lemma}
\begin{proof} See, for instance \cite[Theorem II.9]{newman}.
\end{proof}

The integers $d_1,\ldots,d_r$ determined 
in the previous lemma
are called the \emph{invariant factors} of $A$ and are denoted by $d_1(A),\ldots,d_r(A)$.
Note that $r$ is the rank of the matrix $A$, that is, the smallest $k$ such that $\delta_{k+1}(A)=0$. 
Since the determinantal divisors are invariant under equivalence, we see that 
$$\delta_k(A)=d_1(A)\cdots d_k(A),$$
where we set $d_k(A):=0$ if $k>r$.
Moreover, the vector $(d_1,\ldots,d_r)$ 
is the \emph{type} of the matrix $A$. 
It's worth noticing that, as with the determinantal divisors, the invariant factors are also not multiplicative.

\begin{corollary}\label{cor:equiv-type}
Two integral matrices of same order are equivalent 
if and only if they have the same determinantal divisors, 
which happens if and only if they have the same type. 
\end{corollary}
\begin{proof}
It is clear that having the same determinantal divisors is equal to having the same type, and that
 equivalent integral matrices have the same type.
Now assume that $A_1,A_2\in M(n\times m,\mathbb Z)$ have the same type $(d_1,\ldots,d_r)$. Then, by Lemma \ref{lem:smith-normal-form}, they are both equivalent to the $n\times m$ matrix of the form \eqref{eq:smith-normal-form-matrix} where $D=\diag(d_1,\ldots,d_r)$ and thus are also equivalent to each other.
\end{proof}

\section{Homomorphisms}\label{sec:divisor-homos}

Let $f\:T_1\rightarrow T_2$ be a homomorphism of complex tori. 
Recall that the matrices of the rational representation $\rho_r(f)$ with respect to different bases are equivalent. 
Thus we let $[\rho_r(f)]$ be the matrix of $\rho_r(f)$ with respect to some basis
and define the \emph{$k$-th determinantal divisor}
$\delta_k(f)$, 
the \emph{$k$-th invariant factor} $d_k(f)$ and the \emph{type} of $f$ as those of $[\rho_r(f)]$. 
As before, we have that $d_k(f)|d_{k+1}(f)$ and
$$\delta_k(f)=d_1(f)\cdots d_k(f)$$
for every $1\leq k\leq r$, where $r$ is the rank of $[\rho_r(f)]$.
Note that
$$r=2\dim \im f= 2(\dim T_1-\dim (\ker f)_0).$$
In particular, 
$f$ has finite kernel if and only if $r=2\dim T_1$. 

It's easy to see that the type of a homomorphism is invariant under composition with isomorphisms.

\begin{proposition}\label{prop:type-homo-invariant-iso}
Let 
$f\:T_1\rightarrow T_2$ and
$h\:T_1'\rightarrow T_2'$
be homomorphisms of complex tori. 
If there exist isomorphisms
$h_1\:T_1\rightarrow T_1'$ and
$h_2\:T_2'\rightarrow T_2$ 
such that 
$f=h_2\circ h\circ h_1$, then $f$ and $h$ have the same type.
\end{proposition}
\begin{proof}
Since $f=h_2\circ h\circ h_1$, then
$$[\rho_r(f)]
= [\rho_r(h_2)]\cdot [\rho_r(h)]\cdot [\rho_r(h_1)]$$
and, since $h_1$ and $h_2$ are isomorphisms, then $[\rho_r(h_1)]$ and $[\rho_r(h_2)]$ are invertible as integral matrices. Hence $[\rho_r(f)]$ and $[\rho_r(h)]$ are equivalent and the result follows from Corollary \ref{cor:equiv-type}.
\end{proof}

Now we show that, with a suitable choice of basis, we may assume that the 
matrix of the rational representation of a homomorphism of complex tori is in the Smith normal form.

\begin{proposition}\label{prop:smith-normal-form-homo}
Let $T_1=V_1/L_1$ and $T_2=V_2/L_2$ be complex tori and let
$f\:T_1\rightarrow T_2$ be a homomorphism of type $(d_1,\ldots,d_r)$.
Then there there exist bases of $L_1$ and $L_2$ with respect to which 
\begin{equation}\label{eq:smith-normal-form}
[\rho_r(f)]=\left(\begin{array}{cc}
D & 0\\ 0 & 0
\end{array}\right),
\end{equation}
where $D=\diag (d_1,\ldots,d_r)$.
\end{proposition}
\begin{proof}
Let $A\Pi_1=\Pi_2R$ be the matricial equation associated to $f$, where 
$\Pi_1$ and $\Pi_2$ are period matrices for $T_1$ and $T_2$, respectively, and
$A=[\rho_a(f)]$ and $R=[\rho_r(f)]$, with respect to some choice of basis.
By Lemma \ref{lem:smith-normal-form}, there exist $B\in GL(2g_1,\mathbb Z)$ 
and $B'\in GL(2g_2,\mathbb Z)$ 
such that 
$R=B'R'B$ where $R'$ is of the form \eqref{eq:smith-normal-form-matrix} with $d_k=d_k(f)$. 
But then we have that 
$$A\Pi_1 B^{-1}=\Pi_2 B' R',$$
and since $\Pi_1':=\Pi_1B^{-1}$ and 
$\Pi_2':=\Pi_2 B' $
are period matrices for $T_1$ and $T_2$, respectively,
we see that this is once again the matricial equation for $f$, though with respect to other bases for $L_1$ and $L_2$.
\end{proof}

The next result illustrates the importance of the type, as it shows that 
the
type of a homomorphism completely determines its kernel.

\begin{theorem}\label{teo:type-detrmines-kernel}
Let $f\:T_1\rightarrow T_2$ be a homomorphism of complex tori of type $(d_1,\ldots,d_{r})$. Then 
$$\frac{\ker f}{(\ker f)_0}\cong \mathbb Z_{d_1}\times \cdots \times \mathbb Z_{d_{r}}.
$$
In particular, if $f$ has finite kernel, we have 
${\ker f}\cong \mathbb Z_{d_1}\times \cdots \times \mathbb Z_{d_{r}}.
$
\end{theorem}
\begin{proof}
For $j=1,2$ we set $T_j=V_j/L_j$ and $g_j=\dim T_j$.

Assume first that $\ker f$ is finite, so that $r=2g_1$. 
By Proposition \ref{prop:smith-normal-form-homo},
we may assume that the rational representation of $f$ is in the Smith normal form 
$$[\rho_r(f)]=\left(\begin{array}{c}
D\\0
\end{array}\right),$$
where $D=\diag(d_1,\ldots,d_{2g_1})$,
with respect to  bases
$l_1,\ldots,l_{2g_1}$ and $l_1',\ldots,l_{2g_2}'$ of $L_1$ and $L_2$, respectively.
Note that, since $\ker f$ is finite, we have $g_1\leq g_2$.
Now,  the $j$-th column of $[\rho_r(f)]$ is given by the expression of 
$\rho_a(f)(l_j)$ with respect to the basis $l_1',\ldots,l_{2g_2}'$ of $L_2$, and hence we see that 
$\rho_a(f)(l_j)=d_jl_j'$
for $1\leq j\leq 2g_1$.
This shows that the image of $f$ is of the form $\im f=V'/L'$ where 
$L'$ is the subgroup of $L_2$ generated by $l_1',\ldots,l_{2g_1}'$.
Let $\widetilde f$ be the homomorphism $f$ having $\im f$ as codomain, so 
$\widetilde f\:T_1\rightarrow \im f$
is an isogeny with same kernel as $f$. 
With respect to the isomorphism $L'\cong \mathbb Z^{2g_1}$ given by the basis $l_1',\ldots,l_{2g_1}'$,
we have
$$
\ker f
=\ker \tilde f
\cong \frac{L'}{\rho_a(\tilde f)(L_1)}
\cong \frac{\mathbb Z\times \cdots \times \mathbb Z}{d_1\mathbb Z\times \cdots \times d_{2g_1}\mathbb Z}
$$
and the result follows in this case. 

Now assume $\ker f$ is not finite and let $h=\dim (\ker f)_0$. 
Let $f=g\circ p$ be the Stein factorization of $f$  (see \cite[Sec 1.2]{lange}), given by
$$p\:T_1\rightarrow T_1/(\ker f)_0
\ee
g\:T_1/(\ker f)_0\rightarrow  T_2,
$$
where $p$ is the projection 
and 
$$\ker g=\frac{\ker f}{(\ker f)_0}.$$
In particular, $g$ has finite kernel.
Since $(\ker f)_0$ is a subtorus of $T_1$ then it is of the form
$(\ker f)_0=V''/L''$, where $V''$ is a complex sub-vector space of $V_1$ and $L''=V''\cap L_1$. 
Note that $\dim V''=h$
and let  $l_1,\ldots, l_{2h}$ be a basis of $L''$. 
By \cite[Cor. 3, Chap. I]{cassels}, there exist
$l_1',\ldots, l_{2(g_1-h)}'\in L_1$  such that 
$l_1',\ldots, l_{2(g_1-h)}',l_1,\ldots,l_{2h}$ form a basis of  $L_1$.
Recall that $T_1/(\ker f)_0=(V_1/V'')/(L_1/L'')$ and 
let $\overline l$ be the class of $l$ in the quotient lattice $L_1/L''$,
 for every $l\in L_1$. 
Then 
$\overline l_1',\ldots, \overline l_{2(g_1-h)}'$ is a basis of $L_1/L''$
and, with respect to these bases, we have
$$[\rho_r(p)]=(I_{2(g_1-h)}\;\;0)\in M(2(g_1-h)\times 2g_1,\mathbb Z).$$
Choosing any basis for $L_2$, we consider the matrix $[\rho_r(g)]\in M(2g_2\times  2(g_1-h),\mathbb Z)$ and we have 
$$[\rho_r(f)]=[\rho_r(g)][\rho_r(p)]=[\rho_r(g)](I_{2(g_1-h)}\;\;0)
=([\rho_r(g)]\;\;0).$$
This shows that $f$ and $g$ have the same type, and the result follows from the previous case.
\end{proof}

Its clear from Theorem \ref{teo:type-detrmines-kernel} that the type of a homomorphism of complex tori is a much finer invariant than the degree or the exponent, and that it encompasses both of those invariants.

\subsection{Degree}\label{sec:degree}

Let $f\:T_1\rightarrow T_2$ be a homomorphism of complex tori. 
Recall that the degree of $f$ is defined as the order of its kernel, if finite, and as zero otherwise.
%
%
%
The following result on the degree was already known for isogenies, albeit with a different proof.

\begin{corollary}\label{cor:grauisog}
Let $f\:T_1\rightarrow T_2$ be a homomorphism of complex tori 
and let $g_1=\dim T_1$ and $h=\dim (\ker f)_0$.
Then $\deg(f)=\delta_{2g_1}(f)$ and $\delta_{2(g_1-h)}(f)$ is equal to the number of connected components of $\ker f$.
\end{corollary}
\begin{proof}
If $\ker f$ is finite then the result
follows directly from  Theorem \ref{teo:type-detrmines-kernel}, since 
$\delta_{2g_1}(f)=d_1(f)\cdots d_{2g_1}(f).$
If $\ker f$ is not finite, then the rank of $[\rho_r(f)]$ is smaller than $2g_1$ and we have $\delta_{2g_1}(f)=0=\deg (f)$.
Moreover, by Theorem \ref{teo:type-detrmines-kernel} we have 
$|\ker f/(\ker f)_0|=d_1(f)\cdots d_{r}(f)$
and since $r=2(g_1-h)$, the result follows.
\end{proof}

\subsection{Exponent}

Let $f\:T_1\rightarrow T_2$ be a homomorphism of complex tori having finite kernel. 
Recall that the exponent $e(f)$ of $f$ is the 
exponent of the finite group $\ker f$, that is, $e(f)$ is the smallest $n\in\mathbb Z_{>0}$ such that  $nt=0$ for all $t\in\ker f$.
Recall also that, if $f$ is an isogeny, then the {inverse isogeny} of $f$ is the isogeny $f'\:T_2\rightarrow T_1$ defined by the compositions
$f\circ f'=e(f)$ and $f'\circ f=e(f)$.


\begin{corollary}\label{cor:exponent}
Let $f\:T_1\rightarrow T_2$ be a homomorphism of complex tori having finite kernel and let $g_1=\dim T_1$. Then 
$$e(f)=d_{2g_1}(f) 
\ee \deg(f)=e(f)\cdot \delta_{2g_1-1}(f).$$

Moreover, if $f$ is an isogeny and $f'$ is its inverse isogeny,  
then the type of $f'$ is 
$$(d_1(f'),\ldots,d_{2g_1}(f'))
=\left(1,\frac{d_{2g_1}(f)}{d_{2g_1-1}(f)}, \frac{d_{2g_1}(f)}{d_{2g_1-2}(f)},\ldots,\frac{d_{2g_1}(f)}{d_{1}(f)}\right).$$
In particular, 
$e(f)=d_1(f)\cdot e(f')$.
\end{corollary}
\begin{proof}
The first assertions follow directly from Theorem \ref{teo:type-detrmines-kernel} since $d_k(f)|d_{k+1}(f)$.

Now, if $f'$ is the inverse isogeny of $f$ then $[\rho_r(f')]={e(f)}[\rho_r(f)]^{-1}$, with respect to any choice of basis. By Proposition \ref{prop:smith-normal-form-homo}, we can assume that $[\rho_r(f)]$ is in the Smith normal form and hence 
$[\rho_r(f)]=\diag(d_1(f),\ldots,d_{2g_1}(f))$.
Since $[\rho_r(f)]^{-1}=\diag (\frac1{d_1(f)},\ldots,\frac1{d_{2g}(f)})$, then
$$[\rho_r(f')]=\diag \left(\frac{d_{2g_1}(f)}{d_{1}(f)},\frac{d_{2g_1}(f)}{d_{2}(f)},\ldots, \frac{d_{2g_1}(f)}{d_{2g_1-1}(f)}, \frac{d_{2g_1}(f)}{d_{2g_1}(f)}\right),$$
thus showing the last assertion.
%
\end{proof}

In particular this result recovers the well-known fact that a isogeny $f$ and its inverse isogeny have the same exponent if and only if $f$ is \emph{primitive}, that is, $f$ does not factor through a multiplication by some $n\in\mathbb Z$ with $n\geq 2$. Indeed, $f$ is primitive if and only if the greatest common divisor of the entries of $[\rho_r(f)]$ is equal to 1, that is, 
$\delta_1(f)=1$ or, equivalently, $d_1(f)=1$. Note moreover that the inverse isogeny $f'$ is always primitive, regardless of whether $f$ is or not.

\smallskip

The following result 
shows an alternative, and perhaps more direct, way of computing the exponent of an isogeny and of its inverse isogeny. 

\begin{corollary}\label{cor:expoente2}
Let $f$  be an isogeny of complex tori.
Then $$e(f)=\min \{n\in\mathbb Z_{>0}\ | \ n[\rho_r(f)]^{-1}\text{ is integral}\}.$$

Moreover, if $f'$ is the inverse isogeny of $f$, then 
$$e(f')=\min \left\{n\in\mathbb Z_{>0}\ \left| \ \frac{n}{e(f)}[\rho_r(f)]\text{ is integral}\right.\right\}.$$
\end{corollary}
\begin{proof}
Assume $f\:T_1\rightarrow T_2$  where $T_j=V_j/L_j$ for $j=1,2$ and set $g:=\dim(T_1)=\dim(T_2)$. 
Let us first see that the smallest $n\in\mathbb Z_{>0}$ such that $n[\rho_r(f)]^{-1}$ is integral is independent on the choice of basis. 
Indeed, if $R$ and $R'$ are matrices of $\rho_r(f)$ with respect to distinct choices of bases  for $L_1$ and $L_2$, then 
$R'=B\cdot R\cdot B'$, for some $B,B'\in GL(2g,\mathbb Z)$. 
So, for any $n\in\mathbb Z_{>0}$, if $nR^{-1}$ is integral
then 
$$nR'^{-1}=n(B'^{-1}\cdot R^{-1}\cdot B^{-1})=B'^{-1}\cdot nR^{-1}\cdot B^{-1}$$ 
is integral, since $B^{-1}$ and $B'^{-1}$ are already integral.
Thus we can assume that $R=[\rho_r(f)]$ is in the Smith normal form \eqref{eq:smith-normal-form}.
Then $R=\diag(d_1(f),\ldots,d_{2g}(f))$ so that $R^{-1}=\diag(\frac1{d_1(f)},\ldots,\frac1{d_{2g}(f)})$. 
Since $d_k|d_{k+1}$ for every $1\leq k\leq 2g$, we see that the 
smallest $n\in\mathbb Z_{>0}$ such that $nR^{-1}$ is integral is $n=d_{2g}(f)$, and the result follows from Corollary \ref{cor:exponent}.
\end{proof}

Note that, in the previous result,  the minimum can be replaced by the greatest common divisor. 

\begin{corollary}\label{cor:composition-exponent}
Let $f\:T_1\rightarrow T_2$ and $g\:T_2\rightarrow T_3$ be isogenies of complex tori.
Then $e(g\circ f)$ divides $e(g)\cdot e(f)$.
\end{corollary}
\begin{proof}
Since $[\rho_r(g\circ f)]=[\rho_r(g)]\cdot [\rho_r(f)]$ then 
$[\rho_r(g\circ f)]^{-1}=[\rho_r(f)]^{-1}\cdot [\rho_r(g)]^{-1}$
so that, for $n:=e(g)\cdot e(f)$, the matrix $n[\rho_r(g\circ f)]^{-1}$ is integral.
Hence $n$ is a multiple of $e(g\circ f)$,
 by Corollary \ref{cor:expoente2}.
\end{proof}

\subsection{Type of a polarization}

Let $T=V/L$ be a complex torus of dimension $g$ admiting a polarization $H$. Then there exists a basis of $L$ with respect to which the matrix of the imaginary part $E$ of $H$ is of the form
$$\left(\begin{array}{cc}
0 & D\\-D & 0
\end{array}\right),$$
where $D=\diag(d_1,\ldots,d_g)$ and $d_1,\ldots,d_g\in\mathbb Z_{>0}$ are such that $d_k|d_{k+1}$ for $1\leq k\leq g$.
The integers $d_k$ are unique with this property and we say that 
$(d_1,\ldots,d_g)$ is the \emph{type} and that $e(H):=d_g$ is the \emph{exponent} of the polarization $H$ (see \cite[Sec. 4.1]{lange}).
Note that type of the above matrix, as defined in Section \ref{sec:divisors-matrices}, is $(d_1,d_1,d_2,d_2,\ldots,d_g,d_g)$.
Therefore, since the type of a matrix is preserved by equivalence, 
if $[E]$ is the matrix of $E$ with respect to any choice of basis for $L$, we have that 
$$d_k
=d_{2k}([E])
=\frac{\delta_{2k}([E])}{\delta_{2k-1}([E])}
,$$
for $1\leq k\leq g$, where we set $\delta_0:=1$.
Moreover, recall that $H$ induces an isogeny $\lambda_H\:T\rightarrow\widehat T$ between $T$ and its dual. 
It is well-known that 
$\deg(\lambda_H)=(d_1\cdots d_g)^2$ and $e(\lambda_H)=e(H)=d_g$.  

\begin{proposition}\label{prop:type-polarization}
Let $H$ be a polarization of type $(d_1,\ldots,d_g)$ on a complex torus $T$ of dimension $g$.
Then $d_k=d_{2k}(\lambda_H)$ 
for $1\leq k\leq g$ and, in particular, $e(H)=d_{2g}(\lambda_H)$.
\end{proposition}
\begin{proof}
The statement follows from the fact that  the matrix $[\rho_r(\lambda_H)]$ is equal to the matrix of the imaginary part $E$ of $H$.
\end{proof}

\printbibliography

@book{newman,
	Author = {Newman, Morris},
	Publisher = {Academic Press},
	Title = {Integral matrices},
	edition = {1st edition},
	Year = {1972}}

@book{lange,
	Author = {Birkenhake , Christina and Lange, Herbert},
	Publisher = {Springer},
	Title = {Complex Abelian Varieties - second, augmented edition},
	edition = {2nd edition},
	Year = {2000}}

@book{gantamacher,
	Author = {Gantamacher, F.R.},
	Publisher = {AMS Chelsea Publishing},
	Title = {The theory of matrices - volume 1},
	edition = {1st edition},
	Year = {1959}}

@book{cassels,
	Author = {Cassels, J. W. S.},
	Title = {An Introduction to the Geometry of Numbers},
	edition = {second printing, corrected},
	Publisher = {Springer},
	Year = {1971}}

\bigskip

\noindent{
Juliana Coelho \\ 
Instituto de Matem\'atica e Estat\'istica\\ 
Universidade Federal Fluminense (UFF)\\ 
Rua Prof. Marcos Waldemar de Freitas Reis 
- 24210-201 - Niter\'oi -  RJ,  Brasil}\\
{\smallsl E-mail address: \small\verb?julianacoelhochaves@id.uff.br?}

\end{document}